\documentclass{svproc}
\usepackage{url}

\usepackage{hyperref}
\usepackage{type1cm}        % activate if the above 3 fonts are
\usepackage{makeidx}         % allows index generation
\usepackage{graphicx}        % standard LaTeX graphics tool
\usepackage{multicol}        % used for the two-column index
\usepackage[bottom]{footmisc}% places footnotes at page bottom

\usepackage{newtxtext}       %
\usepackage[varvw]{newtxmath}       % selects Times Roman as basic font

\usepackage{bm}% for bold math symbols
\usepackage{siunitx}% for table alignments
\usepackage{cleveref}
\usepackage[ruled,linesnumbered]{algorithm2e}
\usepackage{microtype} % to solve overfull
\usepackage{color}
\usepackage{lipsum}

\newtheorem{assumption}{Assumption}

\SetKwInput{Input}{Input}
\SetKwInput{Output}{Output}

\crefname{theorem}{theorem}{theorems}
\crefname{proposition}{proposition}{propositions}
\crefname{lemma}{lemma}{lemmas}
\crefname{corollary}{corollary}{corollaries}
\crefname{remark}{remark}{remarks}
\crefname{assumption}{assumption}{assumptions}
\crefname{definition}{definition}{definitions}
\crefname{algocf}{alg.}{algs.}
\Crefname{algocf}{Algorithm}{Algorithms}

\newcommand{\bsx}{{\boldsymbol{x}}}
\newcommand{\bsy}{{\boldsymbol{y}}}
\newcommand{\bsz}{{\boldsymbol{z}}}

\newcommand{\bsnu}{{\boldsymbol{\nu}}}

\newcommand{\N}{{\mathbb{N}}} % natural numbers {1, 2, ...}
\newcommand{\R}{{\mathbb{R}}} % reals
\DeclareSymbolFont{bbold}{U}{bbold}{m}{n}
\DeclareSymbolFontAlphabet{\mathbbold}{bbold}

\newcommand{\calA}{{\mathcal{A}}}

\newcommand{\calF}{{\mathcal{F}}}

\newcommand{\calI}{{\mathcal{I}}}

\newcommand{\calL}{{\mathcal{L}}}

\newcommand{\calO}{{\mathcal{O}}}

\newcommand{\calS}{{\mathcal{S}}}

\newcommand{\calV}{{\mathcal{V}}}

\providecommand{\argmin}{\operatorname*{argmin}}

\definecolor{mygreen}{rgb}{0,0.4,0}

\newcommand{\Exp}[1]{\mathbb{E}\left[#1\right]}

\newcommand{\Prob}[1]{\mathbb{P}\left(#1\right)}

\newcommand{\norm}[1]{\left\lVert#1\right\rVert}
\newcommand{\inp}[2]{\left \langle #1, #2 \right \rangle}
\newcommand{\adj}{\ensuremath{\ast}}
\newcommand{\abs}[1]{\left\lvert#1\right\rvert}
\newcommand{\tol}{\ensuremath{\varepsilon}}
\newcommand{\Deltag}{\ensuremath{\nabla \Delta}} % \lozenge
\newcommand{\work}{\mathrm{Work}}
\newcommand{\qoi}{\mathrm{QoI}}

\begin{document}

\mainmatter              % start of a contribution
\title{A function approximation algorithm using multilevel active subspaces}
\titlerunning{Multilevel Active Subspaces}  % abbreviated title (for running head)
\author{Fabio Nobile\inst{1} \and Matteo Raviola\inst{1}
  \and Ra{\'u}l Tempone\inst{2}\inst{3}}
\authorrunning{Matteo Raviola et al.} % abbreviated author list (for running head)
\tocauthor{Fabio Nobile, Matteo Raviola, and Ra{\'u}l Tempone}
\institute{{\'E}cole polytechnique f{\'e}d{\'e}rale de Lausanne,\\
\email{matteo.raviola@epfl.ch}
\and
King Abdullah University of Science and Technology
\and
Rheinisch-Westf{\"{a}}lische Technische Hochschule Aachen}

\maketitle              % typeset the title of the contribution

\begin{abstract}
  The Active Subspace (AS) method is a widely used technique for identifying the most influential directions in high-dimensional input spaces that affect the output of a computational model.
  The standard AS algorithm requires a sufficient number of gradient evaluations (samples) of the input output map to achieve quasi-optimal reconstruction of the active subspace, which can lead to a significant computational cost if the samples include numerical discretization errors which have to be kept sufficiently small.
  To address this issue, we propose a multilevel version of the AS method (MLAS) that utilizes samples computed with different accuracies and yields different active subspaces across accuracy levels, which can match the accuracy of single-level AS with reduced computational cost, making it suitable for downstream tasks such as function approximation.
  In particular, we propose to perform the latter via optimally-weighted least-squares polynomial approximation in the different active subspaces, and we present an adaptive algorithm to choose dynamically the dimensions of the active subspaces and polynomial spaces.
  We demonstrate the practical viability of the MLAS method with polynomial approximation through numerical experiments based on random partial differential equations (PDEs).
\end{abstract}

\section{Introduction}
The domain of scientific computing has witnessed great advances over the past decades.
The field of uncertainty quantification (UQ) and the study of parametric or stochastic models lead to high-dimensional problems which necessitate analysis and interpretation.
These models, though offering precious insights, come with an attached computational cost that drastically increases with their dimensionality. Reducing the dimensionality while retaining the core information is thus a challenge of paramount importance.
Multiple techniques to tackle this problem have flourished.
Notably, the Reduced Basis Method (RBM) is a classical approach for reducing the dimensionality of the output space in parametric differential models.
It includes, in particular, an offline stage that generates a low-dimensional linear subspace of solution space using high-fidelity snapshots \cite{hesthaven2016certified}.
Other techniques have emerged, on the other hand, to mitigate the dimensionality of the input space in parametric models, so that uncertainty quantification can be achieved with a reduced complexity.
These techniques span nonlinear and linear methods, such as basis adaptation for polynomial chaos expansion \cite{tsilifis2017reduced,tsilifis2018bayesian} and Active Subspaces (AS) method \cite{constantine2015active}.
The latter, in particular, emerged in the field of uncertainty quantification as a dimensionality reduction technique seeking to identify the most influential directions in the input space,
thus greatly helping a variety of downstream tasks on the input-output map such as approximation, integration, and optimization.
While the AS method excels in addressing the dimensionality issue, the associated computational costs for achieving higher accuracy levels remains a concern.

A technique to reduce the computational complexity of parametric/stochastic differential models is the multilevel paradigm which has grown popular in different computational domains.
The paradigm first appeared in \cite{giles2015multilevel} in the context of Monte Carlo approximation (MLMC).
This work provided a framework where different accuracy discretization levels of the underlying differential model could be combined to improve overall efficiency.
The fundamental idea is the use of a hierarchy of discretization levels, where each level provides a different approximation of the mathematical model with coarse discretizations being associated to low query costs, and vice-versa.
The primary objective of the multilevel paradigm is to achieve the same accuracy of the original methodology at a reduced computational cost by taking advantage of the hierarchical structure.
The multilevel paradigm has been successfully employed in other settings such as polynomial regression \cite{hajiali2020multilevel}, filtering \cite{hoel2016multilevel}, and stochastic collocation for random PDEs \cite{haji2016multi}.
Furthermore, it inspired several multifidelity methods, which exploit multiple models at different accuracy, in various fields, from optimization to uncertainty quantification. The authors of \cite{lam2020multifidelity} propose to construct an active subspace from a multifidelity Monte Carlo estimation of the gradient of the model output, however its complexity is limited by the Monte Carlo error.

In this work we present a multilevel version of the Active Subspace method (MLAS) inspired by \cite{hajiali2020multilevel} that builds a different active subspace for each accuracy level, hence retaining the desirable stability property of the AS method and, in turn, not suffering from the Monte Carlo error.
Furthermore, we discuss how to perform the function approximation in the estimated active subspaces by (adaptive) optimally-weighted least-squares polynomial fit.

\section{The Active Subspace method}
\label{nrt-nrt-sec:slas}
In this section, we briefly review the AS method.
We begin by introducing some basic objects and notation.
Consider a Lipschitz domain $\Gamma \subseteq \R^d$ with $d \in \N \cup \{ \infty \}$ and let $\mathcal{B}(\Gamma)$ be the Borel $\sigma$-algebra on $\Gamma$.
Let further $\mu$ be a probability measure on $(\Gamma,\mathcal{B}(\Gamma))$.
For $1 \leq p \leq +\infty$ and $\mathcal{V}$ a Banach space, let us consider the weighted $L^p_\mu(\Gamma;\mathcal{V})$ Lebesgue-Bochner spaces with standard definition, which we will mainly use with $\mathcal{V} = \R$ or $\mathcal{V} = \ell^2(d)$, where $\ell^2(d)$ is the space of sequences $(\mu_1,\dots,\mu_d) \in \R^d$ endowed with the $2$-norm $\norm{\cdot}_{\ell^2(d)}$.
Furthermore, we will omit $\mathcal{V}$ and write $L^p_\mu(\Gamma)$ when no confusion arises.
Finally, let us introduce the weighted Sobolev space $H^1_\mu(\Gamma)$ of squre-integrable functions with square-integrable weak derivative.

We now consider a function $f \in H^1_\mu(\Gamma)$. The goal is to approximate it by first reducing the dimensionality of the input space from $d$ (which is large and possibly infinite) to $r \ll d$.
Then, a second approximation step on the reduced $r$-dimensional space is considered.
The idea behind the AS method is to achieve these two goals by considering a model of the form
\begin{equation}
  \label{nrt-eq:approx-model}
  f \approx g \circ V^\adj,
\end{equation}
where $g: \R^r \to \R$ is a measurable function.
The dimensionality reduction is relegated to a linear operator $\bsy \to V^\adj \bsy$ with $V \in \R^{d \times r}$ orthogonal, $V^\adj$ being the adjoint of $V$.

Let $\mathrm{St}(d,r)$ denote the Stiefel manifold embedded in $\R^{d \times r}$, i.e. the set of orthogonal matrices $V \in \R^{d \times r}$ such that $V^\adj V = I_{r \times r}$. Throughout the paper, slightly abusing notation, we sometimes denote by $V$ the subspace spanned by the columns of the matrix $V \in \mathrm{St}(d,r)$.
Then, $\Pi_V$ denotes the orthogonal projector onto this subspace, that is $\Pi_V \bsy = V V^\adj \bsy$, and $\Pi_V^\perp$ the orthogonal projector onto its orthogonal complement.
In order to approximate $f$ with the model \eqref{nrt-eq:approx-model}, we have to choose both $V$ and $g$.
A natural benchmark is to consider the best approximation in the $L^2_\mu$ sense, namely the optimization problem
\begin{equation}
  \label{nrt-eq:general-AS-obj}
  \min_{
    \substack{\text{$g:\Gamma_V \to \R$ measurable}, \\ V \in \mathrm{St}(d,r)}
  }
  \int_\Gamma (f(\bsy) - g(V^\adj \bsy))^2 d\mu(\bsy).
\end{equation}
Consider $V \in \mathrm{St}(d,r)$ fixed and $W \in \mathrm{St}(d,d-r)$ such that $V V^\adj + W W^\adj = I$, $I$ being the identity map in $\R^d$, so that any $\bsy \in \Gamma$ can be written as
\begin{equation}
  \label{nrt-eq:active-and-inactive-vars}
  \bsy = V \bsx + W \bsz, \qquad \bsx := V^\adj \bsy \in \Gamma_V, \quad \bsz := W^\adj \bsy \in \Gamma_W,
\end{equation}
where $\Gamma_V := V^\adj \Gamma = \{V^\adj \bsy \; : \; \bsy \in \Gamma \}$, and similarly for $\Gamma_W$.
The AS method distinguishes between \textit{active} variables, $\bsx$, and \textit{inactive} ones, $\bsz$.
The active variables are the key parameters or dimensions that exert a significant influence on the output of the mathematical model represented by the function $f$.
Inactive variables, on the other hand, ideally have minimal impact on the system and can be discarded without significantly affecting the accuracy of the model.

Letting $Y$ be a random vector with distribution $\mu$, we can define the marginal measures, denoted as $d\mu_V$ and $d\mu_W$, of the random vectors $X := V^\adj Y$ and $Z := W^\adj Y$, respectively. Define further the conditional measure $d\mu_{Z \vert X}(\cdot \vert \bsx)$ for $Z$ given $X = \bsx$.
It is a well-known result \cite[Chapter 9]{williams_1991} that the optimal solution of Problem \eqref{nrt-eq:general-AS-obj} when performing the minimization over all measurable functions $g:\Gamma_V \to \R$ is given by the conditional expectation
\begin{equation}
  \label{nrt-eq:best-g}
  g^\star(\bsx) = \int_{\Gamma_{\bsx}} f(V \bsx + W \bsz) d\mu_{Z \vert X}(\bsz \vert \bsx),
\end{equation}
where $ \Gamma_{\bsx} := \{ \bsz \in \R^{d-r} \; \vert \; V \bsx + W \bsz \in \Gamma \} \subseteq \Gamma_W $ denotes the support of the measure $\mu_{Z \vert X}(\cdot \vert \bsx)$.
In other words, given $\bsx \in \R^r$, $g^\star(\bsx)$ is the average of $f$ with respect to $\mu$ over all points $\bsy \in \R^d$ which map onto $\bsx$ by $V^\adj$.
Note that, since $f \in L^2_\mu(\Gamma)$, $g^\star$ is unique as an element of $L^2_{\mu_V}(\Gamma_V)$, and we may write
\begin{equation}
  g^\star \circ V^\ast = \Lambda_{L^2_\mu(\Gamma) / V} f,
\end{equation}
where we defined for a given closed subspace $\calS \subseteq L^2_\mu(\Gamma)$ the orthogonal projector $\Lambda_{\calS}: L^2_\mu(\Gamma;\calV) \to \calS \otimes \calV$ by
\begin{equation}
  \label{nrt-eq:L2-proj}
  \Lambda_{\calS} f := \argmin_{v \in \calS \otimes \calV} \norm{f - v}_{L^2_\mu(\Gamma)},
\end{equation}
and $L^2_\mu(\Gamma) / V$ denotes the space of square-integrable functions which are constant with respect to directions in $V$.
We can hence reduce \eqref{nrt-eq:general-AS-obj} to the problem
\begin{equation}
  \label{nrt-eq:approx-err}
  \min_{
    V \in \mathrm{St}(d,r)
  }
  \norm{f - g^\star \circ V^\adj}_{L^2_\mu(\Gamma)}^2 = \int_{\Gamma} (f(\bsy) - g^\star(V^\adj \bsy))^2 d\mu(\bsy).
\end{equation}
Interestingly, the error in this construction can be bounded assuming a Poincaré-type inequality.
\begin{assumption}{(Sliced Poincar\'e inequality)}
  \label{asmp:poincare}
  There exists a constant $C_P > 0$ such that, for any $\bsx \in \Gamma_V$ and $h \in H^1_{\mu_{Z \vert X}(\cdot \vert \bsx)}(\Gamma_{\bsx})$ with zero mean, it holds
  \begin{equation}
    \left( \int_{\Gamma_{\bsx}} h(\bsz)^2 d\mu_{Z \vert X}(\bsz \vert \bsx) \right)^{\frac{1}{2}}
    \leq C_P \left( \int_{\Gamma_{\bsx}} \norm{\nabla h(\bsz)}^2_{\ell^2(d-r)} d\mu_{Z \vert X}(\bsz \vert \bsx) \right)^{\frac{1}{2}}.
  \end{equation}
\end{assumption}
\begin{remark}
  The Poincaré inequality is inherently tied to the domain $\Gamma$, the measure $\mu$, and the subspace $V$.
  The setting we are most interested in is one where the $y_i$'s denote independent and identically distributed random variables, hence we focus on the case where the domain $\Gamma$ is expressed as a Cartesian product space and the measure $\mu$ is of tensor product type, namely $\Gamma = \times_{i=1}^d \Gamma_i$ and $\mu = \otimes_{i=1}^d \mu_i$, where each $\Gamma_i \subseteq \mathbb{R}$ and $\mu_i$ is a univariate measure on $\Gamma_i$.
  The most common settings are the ones where $\mu_i$ is the standard Gaussian measure on $\Gamma_i = \mathbb{R}$, or the uniform measure with $\Gamma_i = [-1,1]$.
  One can show \cite{bebe2003,chen1982} that in the former case $C_P = 1$, while in the latter $C_P$ scales with $\sqrt{d}$.
  These results suggest that the AS method is better suited for the Gaussian case.
\end{remark}
\begin{theorem}{\cite[Theorem 4.3]{constantine2015active}}
  \label{thm:err-bound}
  Let \Cref{asmp:poincare} hold, then, for a fixed $V \in \mathrm{St}(d,r)$, the $L^2_\mu$ error of the approximation $g^\star \circ V^\adj$, where $g^\star$ is defined in \eqref{nrt-eq:best-g}, satisfies
  \begin{equation}
    \label{nrt-eq:err-bound}
    \norm{f - g^\star \circ V^\adj}_{L^2_\mu(\Gamma)} \leq
    C_P \norm{(I-\Pi_{V}) \nabla f}_{L^2_\mu(\Gamma)}.
  \end{equation}
\end{theorem}
This result states that an upper bound for the best reconstruction error of $f$ from the $r$ dimensional subspace $V$ is given by the projection error of the gradient of $f$ onto the subspace $V$.
Instead of minimizing the reconstruction error \eqref{nrt-eq:approx-err} over $\mathrm{St}(d,r)$, owing to Theorem \ref{thm:err-bound} one may minimize the upper bound in \eqref{nrt-eq:err-bound}.
This problem, indeed, turns out to have a simple solution given by the subspace spanned by the first $r$ dominant eigenvectors of
\begin{equation}
  \label{nrt-eq:cov}
  C := \int_\Gamma \nabla f(\bsy) \otimes \nabla f(\bsy) d \mu(\bsy),
\end{equation}
where $\otimes$ denotes the outer product in $\R^d$.
This quantity represents the second moment of $\nabla f$, yet, in accordance with the prevailing terminology in the AS literature, we hereafter refer to it as the covariance.
In practice, in order to compute the active subspace, the integral in \eqref{nrt-eq:cov} must be discretized, which can be done, e.g., using a Monte Carlo estimator based on a random sample $\{\bsy_i\}_{i =1}^M \stackrel{\text{iid}}{\sim} \mu$ of size $M$.
This naturally leads to the straightforward \Cref{algo:SLAS} as the core of the AS methodology. There, $\mathrm{SVD}(C,r)$ denotes the truncated SVD of the matrix $C$ at rank $r$.

Note that, given a rank $r$, to run \Cref{algo:SLAS} one should choose a sample size such that a good reconstruction $\hat{U}_r$ of the optimal subspace $U_r$ is achieved, that is $M$ should be chosen large enough so that $\norm{(\Pi_{U_r}-\Pi_{\hat{U}_r}) \nabla f}_{L^2_\mu(\Gamma)}$ is small enough.
Following \cite{constantine2015active}, we make the following assumption.
\begin{assumption}\label{asmp:rlogr}
  There exists $\eta > 0$ such that, provided $M \geq \calO(\kappa r^\eta \log(r))$ for some $\kappa > 1$,
  \begin{equation}\label{nrt-eq:}
    \norm{(I-\Pi_{\hat{U}_r}) \nabla f}_{L^2_\mu(\Gamma)}
    \lesssim \norm{(I-\Pi_{U_r}) \nabla f}_{L^\infty_\mu(\Gamma)}
  \end{equation}
  holds with probability $1 - \calO(M^{-\kappa})$.
\end{assumption}

\begin{algorithm}[h]
  \caption{Single-level Active Subspace method with ideal reconstruction (\textbf{SLAS}) \cite{constantine2015active}}
  \label{algo:SLAS}
  \Input{Function $f$, subspace dimension $r$, sample size $M$.}

  Draw $\{\bsy_i\}_{i =1}^M \stackrel{\mathrm{iid}}{\sim} \mu$ and compute the sample covariance
  \begin{equation}
    \hat{C} := \frac{1}{M} \sum_{i=1}^M \nabla f(\bsy_i) \otimes \nabla f(\bsy_i).
  \end{equation} \\

  Compute a truncated spectral decomposition $\hat{U}_r, \hat{\Sigma}_r^2 \gets \mathrm{SVD}(\hat{C},r)$. \\

  Compute ideal reconstruction function $g^\star \circ \hat{U}_r^\adj
    \gets  \Lambda_{L^2_\mu(\Gamma) / \hat{U}_r} f$.

  \Output{Approximation $g^\star \circ \hat{U}_r^\adj$.}
\end{algorithm}

\section{Multilevel Active Subspaces}
\label{nrt-nrt-sec:mlas}

In scientific computing applications we often face the impossibility of evaluating the function $f$ and its gradient exactly.
This can be, for instance, due to the fact that the relevant computational model involves the discretization of an ODE or a PDE.
Hence, in practice we have access only to approximations $f_l \approx f$ and, consequently, $\nabla f_l \approx \nabla f$, $l \in \N$ being the level of approximation with $f_\infty := f$.
For instance, in a PDE setting, the level $l$ could be related to a mesh discretization with mesh size $h_l$.
This entails that the larger $l$ is, the higher is the computational cost needed to evaluate $f_l$ and $\nabla f_l$.

A single level AS approximation consists in choosing a level $L \in \N$ and running \Cref{algo:SLAS} with input parameters~$(f_L, r_L, M_L)$.
Denoting the orthogonal matrix computed in the SLAS algorithm by $\hat{U}_L$, let
\begin{equation}
  \mathcal{S}^{\mathrm{SL},\star}_L(f_\infty)
  := g_L^\star \circ \hat{U}_L^\adj
  = \Lambda_{L^2_\mu(\Gamma)/\hat{U}_{L}} f_L,
\end{equation}
be the idealized single-level estimator, which satisfies the single-level error decomposition
\begin{equation}
  \label{nrt-eq:SL-err-dec}
  \norm{f_\infty - \mathcal{S}^{\mathrm{SL},\star}_L(f_\infty)}_{L^2_\mu(\Gamma)}
  \leq \norm{f_\infty - f_L}_{L^2_\mu(\Gamma)}
  + C_P \norm{(I - \hat{\Pi}_L) \nabla f_L}_{L^2_\mu(\Gamma)},
\end{equation}
where we defined the projector $\hat{\Pi}_L := \hat{U}_L(\hat{U}_L)^\adj$.
Note that the above only involves the function at level $L$.
If a sufficiently small error tolerance is requested, $L$ should be chosen large enough to control the first error contribution in \eqref{nrt-eq:SL-err-dec} and many evaluations of $\nabla f_L$ should be utilized for the approximation of the covariance, making the AS method prohibitive in terms of computational cost.
The key observation is that the single-level AS method does not exploit the flexibility we have of querying $f$ at different levels.
Hence, in this section, we aim at designing an improved method which can exploit the whole hierarchy $f_0,f_1,\dots$ of discretizations available.

Our idea is to introduce the multilevel paradigm in the AS method.
The multilevel paradigm is a powerful computational strategy revolving around the use of hierarchies of discretization levels to reduce complexity of computational routines.
A prime example of this approach is the multilevel Monte Carlo (MLMC) method \cite{giles2015multilevel} where the objective is to compute an expectation of the form $\Exp{f_L}$ with reduced complexity.
MLMC is based on the observation that $\Exp{\cdot}$ is a linear operator so that one can decompose the original expectation with the telescoping sum
\begin{equation}
  \Exp{f_L} = \Exp{f_0} + \sum_{l=0}^L \Exp{f_l - f_{l-1}}.
\end{equation}
Then, the idea is to approximate expectations $\Exp{f_l - f_{l-1}}$ separately with a sample size which is large at coarse levels of discretization and small at finer levels, where the differences $\Delta_l := f_l - f_{l-1}$ are small but expensive to evaluate.
In our setting, however, the operator is given by the truncated SVD of the gradient which is nonlinear, hence the telescoping argument cannot be straightforwardly applied.
We take inspiration from \cite{hajiali2020multilevel} which deals with a similar issue.
Given a sequence of ranks $r_0 < \dots < r_L$, our idea is to first build an approximation $f_0 \approx g_0^\star \circ \hat{U}_{L,0}^\adj$ using $M_{L}$ gradient evaluations, that is running \Cref{algo:SLAS} with input parameters $(f_0,r_L,M_L)$.
Then, we correct it with an approximation of $\Delta_1 \approx g_1^\star \circ \hat{U}_{L-1,1}^\adj$ using $M_{{L-1}}$ evaluations running \Cref{algo:SLAS} with input parameters~$(\Delta_1,r_{L-1},M_{L-1})$ and iterate in this manner computing approximations $\Delta_l \approx g_l^\star \circ \hat{U}_{L-l,l}^\adj$ until the finest level $L$ is reached.
We then define the Multilevel Active Subspaces (MLAS) method
\begin{equation}
  \mathcal{S}_L^{\mathrm{ML},\star}(f_\infty)
  := \sum_{l =0}^L g_l^\star \circ \hat{U}_{L-l,l}^\adj,
\end{equation}
where we used the auxiliary definition $f_{-1} := 0$.
Since the sample sizes $M_l$ are linked to the subspace dimensions $r_l$ by \Cref{asmp:rlogr}, it holds that $M_{0} < \dots < M_L$, hence this procedure uses only a few gradient evaluations for differences at finer levels and progressively enlarges the subspace dimension while taking more gradient evaluations at coarser and cheaper levels.
The strategy is summarized in \Cref{algo:ideal-MLAS}.

\begin{algorithm}[h]
  \caption{Multilevel Active Subspaces method (\textbf{MLAS}) with ideal reconstruction}
  \label{algo:ideal-MLAS}
  \Input{Functions $(f_0,\dots,f_L)$, subspaces dimensions $(r_0,...,r_L)$, numbers of samples $(M_0,...,M_L)$.}

  Initialize approximation $\mathcal{S}_L^{\mathrm{ML},\star}(f_\infty) \gets 0$. \\
  \For{$l=0,\dots,L$}{
  Draw $M_{L-l}$ samples $\{\bsy_i\}_{i =1}^{M_{L-l}} \stackrel{\mathrm{iid}}{\sim} \mu$ and compute sample covariance
  \begin{equation}
    \hat{C}_{L-l,l} \gets \frac{1}{M_{L-l}} \sum_{i=1}^{M_{L-l}} \Deltag_l (\bsy_i) \otimes \Deltag_l (\bsy_i).
  \end{equation} \\

  Compute truncated spectral decomposition
  $
    \hat{U}_{L-l,l}, \hat{\Sigma}_{L-l,l}^2 \gets \mathrm{SVD}(\hat{C}_{L-l,l},r_{L-l})
  $
  .\\

  Compute ideal reconstruction function $g_l^\star \circ \hat{U}_{L-l,l}^\adj
    \gets  \Lambda_{L^2_\mu(\Gamma) / \hat{U}_{L-l,l}} f$. \\

  Update approximation $\mathcal{S}_L^{\mathrm{ML},\star}(f_\infty) \gets \mathcal{S}_L^{\mathrm{ML},\star}(f_\infty) + g_l^\star \circ \hat{U}_{L-l,l}$. \\
  }
  \Output{Approximation $\mathcal{S}_L^{\mathrm{ML},\star}(f_\infty)$.}
\end{algorithm}

\subsection{Complexity analysis}
Defining projectors $\hat{\Pi}_{L-l,l} := \hat{U}_{L-l,l} (\hat{U}_{L-l,l})^\adj$ we have the multilevel error decomposition
\begin{equation}
  \label{nrt-eq:ml-err-decomp}
  \norm{f_\infty - \mathcal{S}_L^{\mathrm{ML},\star}(f_\infty)}_{L^2_\mu(\Gamma)}
  \leq \norm{f_\infty - f_L}_{L^2_\mu(\Gamma)}
  + C_P \sum_{l = 0}^L \norm{(I-\hat{\Pi}_{L-l,l}) \Deltag_l}_{L^2_\mu(\Gamma)},
\end{equation}
which reveals a spatial approximation component at level $L$ -- as in the single-level method -- plus a subspace approximation component for the differences at level $l=0,\ldots,L$.
The latter are referred to as \textit{mixed differences} in the literature since they couple spatial and subspace approximation.

In order to analyze the complexity (cost versus accuracy) of the MLAS method, we make some assumptions on the decay of the mixed differences. We introduce a function class $\calF$ of continuous functions on $\Gamma$ endowed with seminorm $\abs{\cdot}_{\calF}$ and define the subspace approximation error at rank $r$ in the function class $\calF$, $r \leq d$, as follows
\begin{equation}
  e_{r} (\calF) := \sup_{f \in \calF} \frac{\norm{(I - \Pi_{U(f)}) \nabla f}_{L^\infty_\mu(\Gamma)}}{\abs{f}_{F_w}},
\end{equation}
where we recall that $U(f) \in \mathrm{St}(d,r)$ denotes the optimal subspace for $\nabla f$ in the $L^2_\mu(\Gamma)$ sense.
This enables us to make an assumption on the rate at which this error decays.
\begin{assumption}{(Strong subspace approximability and convergence of approximations)}
  \label{asmp:ssa}
  It holds
  \begin{equation}
    e_{r}({\calF}) \lesssim r^{-{\alpha}}
  \end{equation}
  for some $\alpha > 0$. Furthermore, given $f_\infty \in {\calF}$ there exist approximations $f_l \in {\calF}$ such that
  \begin{equation}
    \abs{\Delta_l}_{L^2_\mu(\Gamma)} \lesssim h_l^{\beta_w}
    \quad \mathrm{and} \quad
    \abs{\Delta_l}_{\calF} \lesssim h_l^{\beta_s}
  \end{equation}
  for some $0 < \beta_s \leq \beta_w$.
\end{assumption}

As anticipated, since the approximation level $l$ is connected to the granularity $h_l$ of the discretization of the computational model, we consider the setting where the smaller $h_l$ is (i.e., the larger $l$ is) the costlier it is to evaluate $\nabla f_l$.
We call ``$\mathrm{Work}$'' the computational cost of all gradient evaluations (we assume that the cost of the $\mathrm{SVD}$'s is negligible with respect to the cost of collecting all gradient evaluations).
The following assumption makes this intuition precise.
\begin{assumption}{(Sample work)}
  \label{asmp:sw}
  One evaluation of $\nabla f_{l}$ requires a computational effort which is $\work(\nabla f_{l}) \lesssim h_l^{-\gamma}$, for some $\gamma > 0$.
\end{assumption}

Given these assumptions, we can state the following complexity result.
\begin{theorem}{(Idealized multilevel -- complexity)} \label{thm:mlasir-complexity}
  Denote by
  \begin{equation}
    \work(\mathcal{S}_L^{\mathrm{ML},\star}(f_\infty))
    := \sum_{l=0}^L M_{L-l} (\work(\nabla f_l) + \work(\nabla f_{l-1}))
  \end{equation}
  the work that $\mathcal{S}_L^{\mathrm{ML},\star}(f_\infty)$ requires for evaluations of the functions $\nabla f_l$, $l = 0,\ldots,L$. Define
  \begin{equation}
    \rho_{\mathrm{ML}} =
    \begin{cases}
      \frac{\eta}{\alpha},                                           & \text{if $\frac{\gamma}{\beta_s} \leq \frac{\eta}{\alpha}$} \\
      \theta \frac{\gamma}{\beta_s} + (1-\theta) \frac{\eta}{\alpha} & \text{otherwise}
    \end{cases},
  \end{equation}
  where $\theta := \frac{\beta_s}{\beta_w}$ and
  \begin{equation}
    t :=
    \begin{cases}
      2                    & \text{if $\frac{\gamma}{\beta_s} < \frac{\eta}{\alpha}$}                         \\
      3 + \frac{1}{\alpha} & \text{if $\frac{\gamma}{\beta_s} = \frac{\eta}{\alpha}$}                         \\
      1                    & \text{if $\frac{\gamma}{\beta_s} > \frac{\eta}{\alpha}$ and $\beta_w = \beta_s$} \\
      2                    & \text{if $\frac{\gamma}{\beta_s} > \frac{\eta}{\alpha}$ and $\beta_w > \beta_s$} \\
    \end{cases}.
  \end{equation}
  Let $0<\tol\lesssim 1$. Then, there exists $L = L(\tol) \in \N$, $r_0,\dots,r_L$, and $M_0,\dots, M_L$ such that
  \begin{equation}
    \work (\mathcal{S}^{\mathrm{ML}, \star}_L( f_\infty))
    \lesssim \abs{\log(\tol)}^{-t}\log(\abs{\log(\tol)}) \tol^{-\rho_{\mathrm{ML}}}
  \end{equation}
  and such that in an event $E$ with $\Prob{E^c} \lesssim \tol^{\log \abs{\log(\tol)}}$ the multilevel approximation satisfies
  \begin{equation}
    \norm{f_\infty - \mathcal{S}^{\mathrm{ML},\star}_L(f_\infty)}_{L^2_\mu(\Gamma)} \lesssim \tol.
  \end{equation}
\end{theorem}
\begin{proof}
  This proof closely follows that of \cite[Theorem 4.3]{hajiali2020multilevel}.
  The only adjustment needed is to address the fact that the (polynomial) projection in \cite{hajiali2020multilevel} requires $\calO(r \log(r))$ (function) evaluations, whereas our (active subspace) projection requires $\calO(r^\eta \log(r))$ (gradient) evaluations.
\end{proof}

\section{Polynomial approximation}
\subsection{Optimally weighted least squares}
In this section we aim at developing a fully discrete version of the MLAS method for function approximation, giving a practical recipe for the approximation of the idealized reconstruction function \eqref{nrt-eq:best-g}.
In particular, we focus on discrete least-squares fitting based on random evaluations with optimal sampling as proposed in \cite{cohen2017optimal}.
Recall the definition of $L^2_\mu(\Gamma)$-orthogonal projector in \eqref{nrt-eq:L2-proj} and let us fix some space $\calS_m \subseteq L^2_\mu(\Gamma;\R)$ of dimension $m$ of functions defined everywhere on $\Gamma$.
Furthermore, assume that for any $\bsy\in \Gamma$, there exists $v \in \calS_m$ such that $v(\bsy) \neq 0$.
Let $\nu$ be a probability measure on $\Gamma$ such that $\mu \ll \nu$, and consider an i.i.d. random sample $\{ \bsy_j \}_{j=1}^N$ from $\nu$. Then, the least-squares estimator is given by
\begin{equation*}
  \hat{\Lambda}_{\calS_m} g = \mathrm{argmin}_{v \in \calS_m \otimes \calV} \frac{1}{N} \sum_{i=1}^N w(\bsy_i) \norm{v(\bsy_i)-g(\bsy_i)}_\calV^2,
\end{equation*}
where $w := \frac{d\mu}{d\nu}$.
The solution to this least squares problem is provided by the well known normal equations, which are described next.
We first introduce the empirical semi-norm for $\calV$-valued functions $v: \Gamma \to \calV$
\begin{equation}
  \norm{v}_{N,\calV} = \frac{1}{N} \sum_{i=1}^N w(\bsy_i) \norm{v(\bsy_i)}_{\calV}^2,
\end{equation}
which is a Monte Carlo approximation of the $L^2_\mu(\Gamma;\calV)$-norm.
Assume now $\{ \phi_1,\dots,\phi_m \}$ is an $L^2_\mu(\Gamma)$-orthonormal basis of $\calS_m$ and define the matrices $G \in \R^{m \times m}$ and $J \in \calV^m$
\begin{equation}
  G_{ij} = \frac{1}{N} \sum_{k=1}^{N} w(\bsy_k) \phi_i(\bsy_k)\phi_j(\bsy_k) = \inp{\phi_i}{\phi_j}_{N},
  \quad J_i := \inp{g}{\phi_i}_{N}.
\end{equation}
Then, the normal equations read
\begin{equation}
  G c = J,
\end{equation}
where $c \in \calV^m$ is such that $\hat{\Lambda}_{\calS_m} g = \sum_{i=1}^m c_i \phi_i$.
It turns out that, whenever $G$ is well conditioned, the reconstruction error is quasi-optimal, namely it is bounded by the best approximation error
\begin{equation} \label{nrt-eq:poly-best-err}
  e_m(g) = \min_{v \in \calS_m \otimes \calV} \norm{g-v}
\end{equation}
up to a constant term.

Then, the idea is to choose the sample size $N$ big enough so that $\norm{G-I} < \delta$ with high probability.
Let us introduce the function
\begin{equation}
  \bsy \mapsto k_{m,w}(\bsy) = \sum_{j=1}^m w(\bsy) \phi_j(\bsy)^2,
\end{equation}
whose reciprocal is known as the Christoffel function \cite{nevai1986geza}, and define
\begin{equation}
  K_{m,w}:= \norm{k_{m,w}}_{L^\infty(\Gamma)},
\end{equation}
which trivially satisfies $K_{m,w} \geq m$ since $\int_\Gamma k_{m,w} d \mu = m$.
We can now state the following theorem.
\begin{theorem}{(From \cite{cohen2017optimal})}
  For any $t>0$, if $m$ and $N$ are such that
  \begin{equation}
    K_{m,w} \leq \kappa \frac{N}{\log(N)}, \quad \mathrm{with} \quad \kappa =\frac{1- \log(2)}{2+2t},
  \end{equation}
  then the weighted least square estimator satisfies
  \begin{equation}
    \Exp{\norm{g-\hat{\Lambda}_{\calS_m} g}^2_{L^2_\mu(\Gamma;\calV)}} \leq (1 + \epsilon(s)) e_m(g)^2,
  \end{equation}
  where $e_m(g)$ is defined in \eqref{nrt-eq:poly-best-err}, with probability at least $1-N^{-t}$.
  \label{thm:theorem-approx-lstsq}
\end{theorem}

The above theorem suggests that, to reduce as much as possible the sample size, one should choose the sampling distribution $\nu$ that minimizes $K_{m,w}$.
In this sense, the optimal sampling measure $\nu^\star$ is given by
\begin{equation}
  \frac{1}{w^\star} := \frac{d \nu^\star}{d \mu}
  = \frac{1}{m} \sum_{j=1}^m \phi_j(y)^2,
\end{equation}
which yields the optimal $K_{m,w^\star} = m$.

\begin{example}
  Let $\Gamma=\R$ with Gaussian density $d \mu$, $\calV = \R$, and consider the space of polynomials of degree smaller than $m$, that is $\calS_m = \mathbb{P}_{m-1}$, then we may take $\phi_j = H_{j-1}$, $H_j$ being the $j$-th normalized Hermite polynomial.
  When the sampling measure $d\nu = d\mu$ and $w(x)=1$, it holds that $K_{m,w}$ is unbounded.
  However, if we consider the optimal sampling measure $\frac{1}{m} \sum_{j=0}^{m-1} H_j(y)^2$, drawing $\calO(m \log(m))$ samples is sufficient to achieve quasi-optimal reconstruction error.
\end{example}

\subsection{Polynomial approximation in the active subspaces}
We now turn to the combination of the AS method and weighted least-squares for the full approximation of the target function $f:\Gamma \to \R$.
For simplicity, here we consider the case where $\mu$ is tensorized Gaussian, so that \Cref{asmp:poincare} holds with $C_P = 1$, the projected measure $\mu_V$ on the active variables is a tensorized Gaussian for any $V \in \mathrm{St}(d,r)$, and the conditional measure $\mu_{Z \vert X}(\cdot \vert \bsx)$ is tensorized Gaussian independent of $\bsx \in \Gamma_V$ and coincides with the projected measure $\mu_W$ on the orthogonal to the active subspace, where we defined $W \in \mathrm{St}(d,d-r)$ such that $W W^\adj = (I - V V^\adj)$.
Note that, in the following, we use $W$ repeatedly to denote the subspace corresponding to the inactive variables.
Introduce now the set of multi indices $\calL := \{ \bsnu \in \N^d \,:\, \sup_{j \,:\, \nu_j \neq 0} j < \infty \}$, $\calL_r := \{ \bsnu \in \calL \,:\, \nu_j = 0 \, \forall \, j>r \}$, and $\calL_{>r} := \{ \bsnu \in \calL \,:\, \exists \, j>r \, \mathrm{s.t.}\, \nu_j \neq 0\}$.
We can then define the tensorized Hermite polynomial in the rotated variables $(\bsx, \bsz)$ associated to a multi index $\bsnu \in \calL_r$ as
\begin{equation}
  H_\bsnu(\bsx,\bsz) := \Pi_{j = 1}^r H_{\nu_j}(x_j),
\end{equation}
$H_{\nu}$ being the one-dimensional normalized Hermite polynomial of degree $\nu \in \N$.
Let us introduce a polynomial space $P_{\Xi_{m,r}}$  of dimension $m$ spanned by polynomials constant in the inactive variables, that is  $P_{\Xi_{m,r}} := \mathrm{span}\{ H_{\bsnu} \,:\, \nu \in \Xi_{m,r} \}$ where $\Xi_{m,r} \subset \calL_r$ has cardinality $m$.
Note that for such a polynomial space the optimal sampling measure is given by
\begin{equation}
  \frac{d \nu^\star}{d \mu}(\bsx, \bsz)
  = \frac{d \nu_V^\star}{d \mu_V}(\bsx)
\end{equation}
since the polynomials are constant with respect to the inactive variables.
Hence, once the active subspace $\hat{U}_L$ has been computed as in \Cref{algo:SLAS}, we further draw a pair of i.i.d. samples $\{ \bsx_j \}_{j=1}^{N_L}$ and $\{ \bsz_j \}_{j=1}^{N_L}$ from the optimal measure $\nu_V^\star$ and the Gaussian measure $\mu_W$, respectively.
Then, defining $\hat{W}_L \in \mathrm{St}(d,d-r)$ such that $\hat{W}_L \hat{W}_L^\adj = (I - \hat{U}_L \hat{U}_L^\adj)$, the single-level estimator is defined by
\begin{equation}
  \begin{split}
    \calS_L^{\mathrm{SL}}(f)
    := & \left( \mathrm{argmin}_{v \in P_{\Xi_{m,r}}} \frac{1}{N} \sum_{i=1}^N w(\bsx_i) \left(v(\bsx_i)-f(\hat{U}_L \bsx_i + \hat{W}_L \bsz_i)\right)^2 \right) \circ \hat{U}_L^\adj \\
    =  & \left( \hat{\Lambda}_{P_{\Xi_{m,r}}} f \right) \circ \hat{U}_L^\adj.
  \end{split}
\end{equation}
The single-level error decomposition
\begin{equation}\label{nrt-eq:SLASPA-err-dec}
  \begin{split}
    \norm{f_\infty - \mathcal{S}^{\mathrm{SL}}_L(f_\infty)}_{L^2_\mu(\Gamma)}
     & \leq \norm{f_\infty - f_L}_{L^2_\mu(\Gamma)}
    + C_P \norm{(I - \hat{\Pi}_L) \nabla f_L}_{L^2_\mu(\Gamma)}                                            \\
     & + \norm{g_L^\star -  \hat{\Lambda}_{P_{\Xi_{m,r}}} f_L}_{L^2_{\mu_{\hat{U}_L}}(\Gamma_{\hat{U}_L})}
  \end{split}
\end{equation}
reveals the fact that three errors should be balanced: the spatial approximation error, the SVD error of the active subspace, and the polynomial approximation error in the active subspace.

For the multilevel strategy, we simply repeat the procedure just described at each level of approximation.
More precisely, let us consider a sequence of ranks $r_0 < r_1 < \dots<r_L$ and a sequence of polynomial spaces of sizes $m_0 < m_1 < \dots < m_L$.
Then, once the active subspace $\hat{U}_{L-l,l}$ at level $l$ of size $r_{L-l}$ has been computed as in \Cref{algo:SLAS}, we further draw a pair of $N_{L-l} = \calO(m_{L-l} \log(m_{L-l}))$ i.i.d. samples $\{ \bsx_j \}_{j=1}^{N_{L-l}}$ and $\{ \bsz_j \}_{j=1}^{N_{L-l}}$ from the optimal measure $\nu_V^\star$ and the Gaussian measure $\mu_W$, respectively.
This defines the empirical projector $\hat{\Lambda}_{L-l,L-l} := \hat{\Lambda}_{P_{\Xi_{m_{L-l},r_{L-l}}}}$.
Then, we define the multilevel estimator as
\begin{equation}
  \mathcal{S}_L^{\mathrm{ML}}(f_\infty)
  := \sum_{l =0}^L \left( \hat{\Lambda}_{L-l,L-l} \Delta_{l} \right) \circ \hat{U}_{L-l,l}^\adj,
\end{equation}
with $f_{-1} := 0$.
The full procedure is detailed in \Cref{algo:MLASPA}.
The multilevel error decomposition then reads
\begin{equation}
  \label{nrt-eq:mlaspa-err-decomp}
  \begin{split}
    \norm{f_\infty - \mathcal{S}_L^{\mathrm{ML}}(f_\infty)}_{L^2_\mu(\Gamma)}
     & \leq \norm{f_\infty - f_L}_{L^2_\mu(\Gamma)}
    + C_P \sum_{l = 0}^L \norm{(I-\hat{\Pi}_{L-l,l}) \Deltag_l}_{L^2_\mu(\Gamma)}                                         \\
     & + \sum_{l=0}^L \norm{g_l^\star -  \hat{\Lambda}_{{L-l,L-l}} \Delta_l}_{L^2_{\mu_{\hat{U}_L}}(\Gamma_{\hat{U}_L})}.
  \end{split}
\end{equation}

\begin{algorithm}[h]
  \caption{Multilevel Active Subspaces method with polynomial approximation (\textbf{MLASPA})}
  \label{algo:MLASPA}
  \Input{Functions $(f_0,\dots,f_L)$, subspaces dimensions $(r_0,...,r_L)$, numbers of samples $(M_0,...,M_L)$, index sets $(\Xi_{m_0, r_0}, \dots,\Xi_{m_L,r_L})$}

  Initialize approximation $\mathcal{S}_L^{\mathrm{ML}}(f_\infty) \gets 0$. \\
  \For{$l=0,\dots,L$}{
  Draw $M_{L-l}$ samples $\{\bsy_i\}_{i =1}^{M_{L-l}} \stackrel{\mathrm{iid}}{\sim} \mu$ and compute sample covariance
  \begin{equation}
    \hat{C}_{L-l,l} \gets \frac{1}{M_{L-l}} \sum_{i=1}^{M_{L-l}} \Deltag_l (\bsy_i) \otimes \Deltag_l (\bsy_i).
  \end{equation} \\
  Compute truncated spectral decomposition $
    \hat{U}_{L-l,l}, \hat{\Sigma}_{L-l,l}^2 \gets \mathrm{SVD}(\hat{C}_{L-l,l},r_{L-l})
  $. \\

  Draw a pair of $N_{L-l} = \calO(m_{L-l} \log(m_{L-l}))$ i.i.d. samples $\{ \bsx_j \}_{j=1}^{N_{L-l}}$ and $\{ \bsz_j \}_{j=1}^{N_{L-l}}$ from the optimal measure $\nu_{\hat{U}_{L-l,l}}^\star$ and the Gaussian measure $\mu_{\hat{W}_{L-l,l}}$. \\

  Solve the discrete least squares problem
  \begin{equation}
    g_l
    \gets \argmin_{v \in P_{\Xi_{m_{L-l},r_{L-l}}}} \frac{1}{N_{L-l}} \sum_{i=1}^{N_{L-l}} w(\bsx_i) \left(\Delta_l(\hat{U}_{L-l,l} \bsx_i + \hat{W}_{L-l,l} \bsz_i) - v(\bsx_i)\right)^2.
  \end{equation} \\

  Update approximation $\mathcal{S}_L^{\mathrm{ML}}(f_\infty) \gets \mathcal{S}_L^{\mathrm{ML}}(f_\infty) + g_l \circ \hat{U}_{L-l,l}$. \\
  }
  \Output{Approximation $\mathcal{S}_L^{\mathrm{ML}}(f_\infty)$.}
\end{algorithm}

\subsection{An adaptive algorithm} \label{nrt-nrt-sec:amlaspa}
In order to choose $L$, $(r_0,\dots,r_L)$, $(M_0,\dots,M_L)$, and $(\Xi_{m_0, r_0}, \dots,\Xi_{m_L,r_L})$ to run \Cref{algo:MLASPA} -- that is, in order to balance all error contributions involved -- we need to know a priori estimates for the strong subspace approximation error, the spatial discretization error associated to $f_{\infty} \approx f_l$, the cost per evaluation of $f_l$ and $\nabla f_l$, and we need to be able to design an adequate sequence of polynomial spaces to use in each subspace $\hat{U}_{L-l,l}$.
This is not always possible in practice, hence, in this section we introduce an adaptive Multilevel Active Subspaces algorithm with polynomial approximation (AMLASPA), which does not require the knowledge of the aforementioned information.

As far as the polynomial approximation is concerned, we assume we have a routine $\texttt{Polyapprox}(h,V,\mathrm{tol})$ which, given a function $h \in L^2_\mu(\Gamma)$, an orthogonal matrix $V \in \mathrm{St}(r,d)$, and a tolerance $\mathrm{tol}$, adaptively builds and returns a polynomial approximant of $\Lambda_{L^2_\mu(\Gamma)/V} h$ on the $r$ variables defined by $\bsx = V^\adj \bsy$ whose $L^2_{\mu_V}(\Gamma_V)$-error is below the prescribed tolerance $\mathrm{tol}$, along with the work needed for its construction (e.g., proportional to the number of evaluations of $h$ used to fit the polynomial).
Furthermore, we assume that \texttt{Polyapprox} can take a warm start, that is, we can start building the polynomial approximant from an existing index set.
See \cite{hajiali2020multilevel} for more details on how such a function may be constructed.
Then, when we use this function in the AMLASPA, if we have already built a polynomial approximant based on parameters $(\Delta_l,V,\mathrm{tol})$ at an earlier iteration, and we wish to build a polynomial approximant with parameters $(\Delta_l,V',\mathrm{tol}')$ with $\mathrm{dim}(V')>\mathrm{dim}(V)$ and $\mathrm{tol}' < \mathrm{tol}$, we assume we build the polynomial approximant starting from the old one as initial condition, so that we do not waste any computational effort.

To construct the sequence of active subspaces in an adaptive fashion, the algorithm constructs a (finite) downward-closed index set $\calI \subset \N^2$.
Given such a set and chosen a priori a function $r:\N \to \N$, we let $\hat{V}_{k,l}$ be the subspace spanned by the eigenvectors of the empirical covariance of $\Deltag_l$ from position $r(k-1)+1$ to $r(k)$ and $\hat{U}_{k,l} := \cup_{i=0}^k \hat{V}_{k,l}$.
Let further
\begin{equation}
  \mathcal{A}(k,l) :=
  \begin{cases}
    \{(k+1,l)\}                         & \text{if $(k,l) \neq (0,l_{\mathrm{max}})$} \\
    \{(k+1,l), (0,l_{\mathrm{max}}+1)\} & \text{if $(k,l) = (0,l_{\mathrm{max}})$}    \\
  \end{cases}
\end{equation}
denote the set of admissible multi-indices that are neighboring $(k,l) \in \calI$, where $l_{\mathrm{max}} := \max \{l \,:\, (0,l) \in \calI\}$.
For each multi-index $(k,l) \in \calI$, we estimate the norm of the projection of $\nabla \Delta_l$ onto $\hat{V}_{k,l}$ via $\sqrt{\sum_{j=r(k-1)+1}^{r(k)} \hat{\sigma}^2_{l,j}}$, which we also use to update an estimate of the $\mathrm{SVD}$ error at level $l$.
This norm represents the gain that was made by adding $(k,l)$ to $\calI$.
Furthermore, we estimate the $\mathrm{SVD}$ work that adding this multi-index incurred.
The work, which we denote through a function $\work: \N^2 \to \R$, could be estimated directly using a timing function, or it can be based on a work model, e.g. the product of work per gradient evaluation times number of evaluations needed.
Then, we call $\texttt{Polyapprox}$ to perform polynomial approximation at level $l$ on the variables defined by $\hat{U}_{k,l}$, with tolerance matching the SVD error.
Note that in the work associated to $(k,l)$ we add the extra work needed to fit the polynomial (we say extra to allow for a warm start of the polynomial space).
Furthermore, we shall choose a function for the number of samples $M:\N \to \R$ by, e.g., linking it to the rank function.
We suggest as in \cite{constantine2015active} to use $M(l) = C r(l) \log(r(l)+1)$ for some constant $C>1$.
With these ingredients, we can construct $\calI$ similarly to \cite{hajiali2020multilevel}.
We simply start with $\calI = \{(0,0)\}$, then, for every iteration of our algorithm, we find the index $(k,l) \in \calI$ which has a non-empty set $\calA(k,l)$ of admissible neighbors and which maximizes the ratio between the gain and work estimates.
Finally, we add those neighbors to the set $\calI$ and repeat.

\section{Numerical experiments}
Let us consider the following linear elliptic second order PDE
\begin{equation}
  \label{nrt-eq:model-prob}
  \begin{aligned}
    -\mathrm{div}(a \nabla u) & = h &  & \text{on $D \subset \R^n$}, \\
    u                         & = 0 &  & \text{on $\partial D$},
  \end{aligned}
\end{equation}
where $a:D \times \Gamma \to \R$ is the diffusion coefficient.
We assume that the diffusion coefficient $a$ is a random function defined as $a = \exp(b)$, where $b$ is a centered Gaussian process defined on $D$.
In particular, we focus on the case where $\Gamma = \R^d$, $a = a(\bsy) = \exp(b(\bsy))$ and $b(\bsy)$ is expanded in parametric form as
\begin{equation}
  \label{nrt-eq:diffusion-aff-exp}
  b(\bsy) = \bar{b} + \sum_{j=1}^d y_j \psi_j,
\end{equation}
where the $y_j$'s are standard normal random variables and the functions $\psi_j: D \to \R$ are given.
Let us remark that here $d \in \N \cup \{\infty\}$, that is the number of parameters $\bsy$ in the model, may in theory be infinite.
For any given $\bsy \in \Gamma$ such that $b(\bsy)\in L^\infty(D)$, Lax-Milgram theory allows us to define the solution $u(\bsy)$ in the space $V \coloneq H^1_0(D)$ through the variational formulation
\begin{equation}
  \int_D a(\bsy) \nabla u \cdot \nabla v = \int_D h v, \qquad v \in V.
\end{equation}
The solution map is only defined on the set
\begin{equation}
  \Gamma_0 \coloneq \{\bsy \in \Gamma \; : \; b(\bsy)\in L^\infty(D) \},
\end{equation}
which is equal to $\Gamma$ in the case of finitely many variables, that is, when $\psi _j=0$ for $j>J$ for some $J \in \N$.
However, in both finitely or countably many variable cases, the solution map is not uniformly bounded.
\begin{theorem}{\cite{bachmayr2017sparse}} \label{thm:logn-well-posed}
  Assume that there exists a sequence $(\rho _j)_{j\geq 1}$ of positive numbers such that
  \begin{equation}
    \sup _{x\in D}\sum _{j\geq 1} \rho _j |\psi _j(x)| =K <\infty,
  \end{equation}
  holds and $\sum _{j\geq 1} \exp (-\rho _j^2)<\infty$.
  Then, $\Gamma_0$ has full measure and $\mathbb{E}(\exp (k\|b\|_{L^\infty }))<\infty$ for all $k<\infty$.
  In turn, the map $\bsy \mapsto u(\bsy)$ belongs to the Bochner space $L^k_\mu(\Gamma;V)$ for all $k<\infty$.
\end{theorem}

Introduce now a smooth quantity of interest $\qoi: V \to \R$ and denote $\qoi'(u)$ its Fréchet differential at $u \in V$.
Then, we have that the function of interest is
\begin{equation}
  f(\bsy) := \qoi(u(\bsy)).
\end{equation}
Note that it is possible to show that the gradient of $f$ at $\bsy \in \Gamma$ is in $L^2_\mu(\Gamma;\ell^2(d))$ \cite{bachmayr2017sparse}.

We ran our experiments based on model \eqref{nrt-eq:model-prob}-\eqref{nrt-eq:diffusion-aff-exp} where $h=0$, $D = [0,1]^2$, $\bar{b} = -4.6$, $\psi_{2j} = \frac{1}{j^\alpha} \cos(j \pi x_1)$, $\psi_{2j+1} = \frac{1}{j^\alpha} \sin(j \pi x_2)$, and $\qoi$ is the spatial average.

In the first pair of plots in \Cref{fig:proj-err}, we show a decay in the projection error, indicating that this example provides fertile ground for applying the AS technique.
In the right plot, which illustrates the projection error of the gradient of the differences, we see that increasing the level of approximation leads to a decrease in error.
However, the rate of decay with respect to the rank remains consistent across levels, meaning that $\nabla f$ features mixed regularity in physical and parameter space. Interestingly, the rate of decay with respect to the rank deteriorates when transitioning from functions to differences, indicating that the differences feature less regularity that the functions themselves.

In the second pair of plots, the left plot demonstrates the computational complexity of SLAS and MLAS under ideal reconstruction (\Cref{algo:SLAS,algo:ideal-MLAS}, respectively), estimated with a model where the computational work is proportional to the number of evaluations of $\nabla \Delta_l$ and $\Delta_l$ used for the computation of the active subspaces and the polynomial approximations, respectively.
The complexity of MLAS aligns with the predictions of \Cref{thm:mlasir-complexity}, showing it to be smaller than that of SLAS.
The right plot presents the complexity of adaptive polynomial approximation algorithms.
Again, the multilevel strategy outperforms the single-level approach, as well as the multilevel polynomial approximation algorithm from \cite{hajiali2020multilevel}.

\begin{figure}[h]
  \centering
  \includegraphics[width=0.45\textwidth]{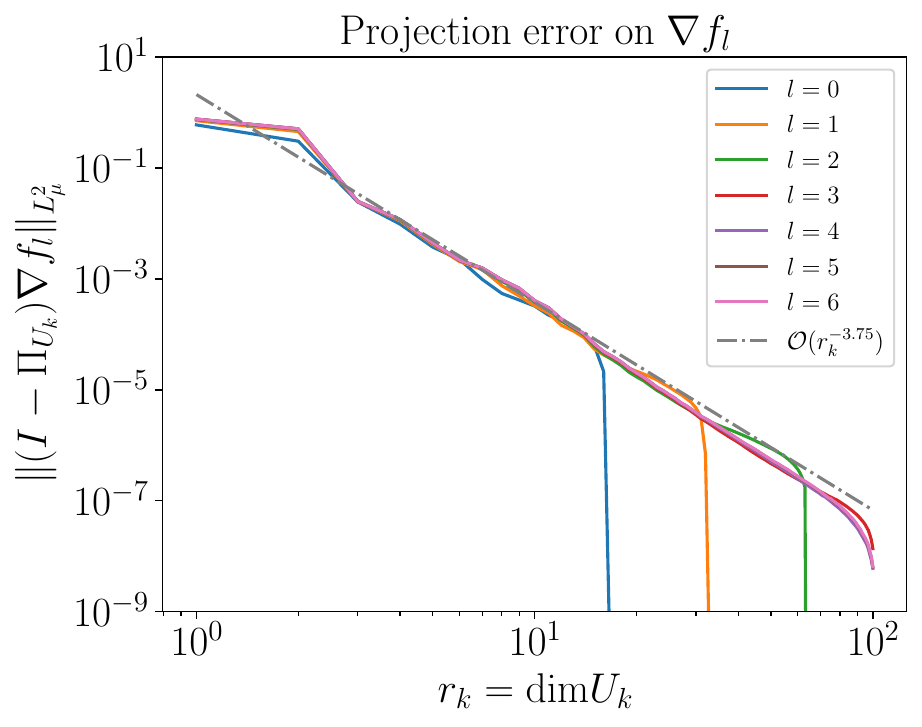}
  \includegraphics[width=0.45\textwidth]{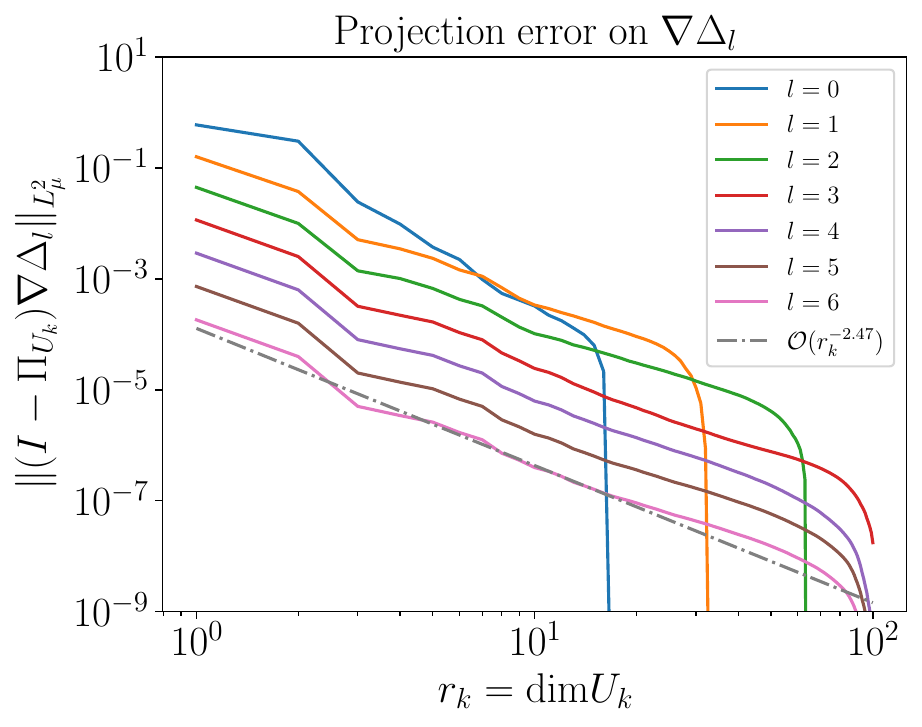}
  \caption{Projection error for $d=100$ and $\alpha = 2$ across different levels for $\nabla f_l$ (left) and $\nabla \Delta_l$ (right).}
  \label{fig:proj-err}
\end{figure}

\begin{figure}[h]
  \centering
  \includegraphics[width=0.45\textwidth]{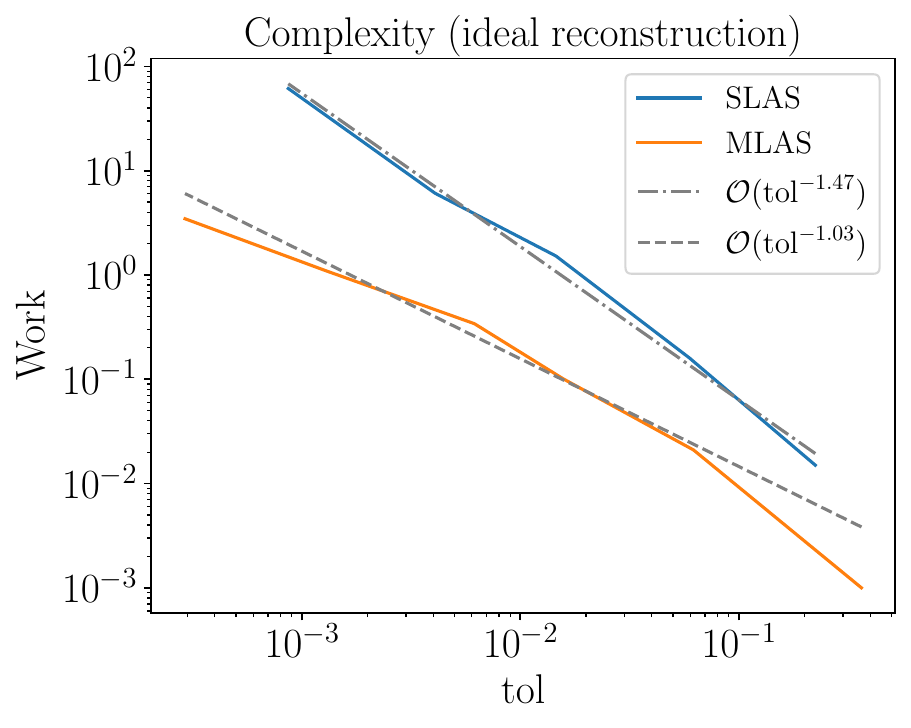}
  \includegraphics[width=0.47\textwidth]{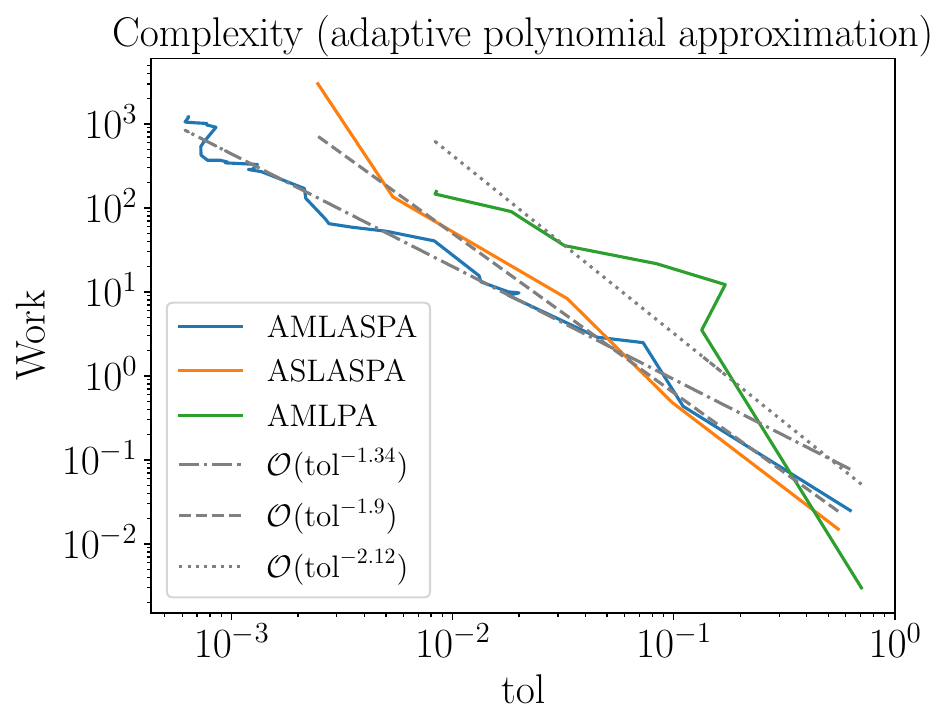}
  \caption{Left: Complexity of \Cref{algo:SLAS,algo:ideal-MLAS}. Right: complexities of single-level AS with polynomial approximation (ASLASPA), Algorithm AMLASPA described in Section \ref{nrt-nrt-sec:amlaspa}, and the adaptive multilevel polynomial approximation algorithm (AMLPA) proposed in \cite{hajiali2020multilevel}.}
\end{figure}

\section{Conclusions}
In this work, we introduced the Multilevel Active Subspaces (MLAS) method, a novel framework that combines the dimensionality reduction capabilities of the Active Subspace (AS) approach with the computational efficiency of the multilevel paradigm.
The MLAS method is designed for scenarios where approximations $f_l \approx f$ of a target map $f$ are available at varying discretization levels $l$, with the cost of evaluating $f_l$ and its gradient $\nabla f_l$ increasing with $l$.
By hierarchically distributing computational resources across discretization levels, MLAS achieves efficiency through a level-dependent reduction in active subspace dimension: at finer levels, where evaluation costs are highest, smaller active subspaces are constructed to minimize the number of gradient evaluations required.
We conducted a theoretical analysis of the MLAS algorithm with ideal reconstruction functions in the active subspaces, demonstrating that it reduces computational complexity compared to single-level AS methods, provided that the maps $f_l$ satisfy certain smoothness and regularity assumptions.
We also proposed a practical algorithm, which constructs polynomial approximations of the ideal reconstruction functions on the active subspaces.
This approach employs optimally weighted discrete least-squares, yielding a fully implementable algorithm for multilevel function approximation.
An adaptive version of this algorithm was also proposed, enabling the simultaneous construction of active subspaces and polynomial approximants at each level without requiring explicit knowledge of model parameters.
Numerical experiments validated the effectiveness of MLAS on a linear elliptic PDE with log-normal random diffusion coefficients.
These experiments demonstrated the method's ability to significantly reduce computational costs while maintaining accuracy, outperforming single-level methods.
The results confirm the applicability of MLAS to parametric PDEs with sufficiently smooth solution maps, showcasing its potential as a tool for high-dimensional approximation problems.

%%% To ensure the bibliography has the correct style please run bibtex with the spmpsci style
%%% which is included in the Springer zip file
%%% Here we assume refs.bib would be the name of the bib file containing bibliographic info
%%% You can then copy the .bbl produced, as given in the example below
%%%
% Desktop path
\bibliographystyle{spmpsci}
\bibliography{biblio.bib}

\end{document}